&amslplain

\documentclass[a4paper,reqno]{amsart}

\usepackage{amssymb}

\usepackage{hyperref}
\hypersetup{
    colorlinks=true,
    linkcolor=blue,
    citecolor=blue,
    urlcolor=blue,
}

\newtheorem{theorem}{Theorem}

\newtheorem{lemma}[theorem]{Lemma}
\newtheorem{proposition}[theorem]{Proposition}

\theoremstyle{remark}
\newtheorem*{remark}{Remark}

\numberwithin{equation}{section}
\numberwithin{theorem}{section}

\def\N{{\mathbb N}}

\def\R{{\mathbb R}}

\def\Z{{\mathbb Z}}

\def\B{{\mathcal B}}

\def\C{{\mathcal C}}

\def\G{{\mathcal G}}

\def\I{{\mathcal I}}

\def\V{{\mathcal V}}

\def\1{{\bf 1}}

\def\var{{\rm var}}

\begin{document}

\title{Unique Bernoulli Gibbs states and $\text{g}$-measures}

\author{Paul Hulse}

\thanks{Email: phulse@ed.ac.uk}

\begin{abstract}
A sufficient condition for the Gibbs states
of a shift-invariant specification on a one-dimensional lattice
to be the $g$-chains for some continuous function $g$ is obtained.
This is then used to derive criteria under which there is a unique Gibbs state,
which is also shift-invariant and Bernoulli.
\end{abstract}

\maketitle

\section{Introduction}
We consider Gibbs states on sequence spaces, that is, probability measures
for which there is a specified set of functions that determine
the probabilities of configurations on any finite subset
conditioned on the configuration outside that subset.
Let $X=S^I$ be a sequence space, where
$S$ is a finite set with the discrete topology,
and $I$ is a countable (infinite) set.
The topology on $X$ is the product topology,
with respect to which $X$ is compact,
and $\B$ denotes the corresponding Borel $\sigma$-algebra.
If $x\in S^{\Lambda^\prime}$ $(\Lambda^\prime\subseteq I)$
and $\Lambda\subseteq\Lambda^\prime$,
then $x_\Lambda$ denotes the natural projection of $x$ onto $S^\Lambda$,
and $[x]_\Lambda$ is the cylinder set $\{y\in X:y_\Lambda=x_\Lambda\}$;
similarly, if $s\in S$ and $i\in I$, $[s]_i$ denotes $\{y\in X:y_i=s\}$.
The sub-$\sigma$-algebra of ${\mathcal B}$ generated
by $\{[x]_\Lambda\,:\, x\in X\}$ is denoted by $\B_\Lambda$,
and $\C$ denotes the set of finite, non-empty subsets of $I$.

A {\em specification} on $X$ is a collection
$\V=\{\phi_\Lambda\}_{\Lambda\in\C}$
of non-negative Borel functions on $X$ such that, for all
$\Lambda\in\C$ and $x\in X$,
\begin{equation}
\sum_{y\in[x]_{\Lambda^c}}\phi_\Lambda(y)=1,
\label{1ai}
\end{equation}
and for $K\subseteq\Lambda$,
\begin{equation}
\phi_\Lambda(x)=\phi_K(x)\sum_{y\in[x]_{K^c}}\phi_\Lambda(y).
\label{1aii}
\end{equation}
The specification $\V$ is said to be continuous if 
$\phi_\Lambda$ is continuous for all $\Lambda\in\C$,
and positive if $\phi_\Lambda>0$ for all $\Lambda\in\C$.
If $I=\Z^d$ $(d\geq1)$ then $T=\{T_i\}_{i\in I}$ denotes the
$d$-dimensional shift on $X$, that is,
$(T_ix)_j=x_{i+j}$ $(x\in X,\ i,j\in\Z^d)$,
and $\V$ is said to be $T$-invariant if
$\phi_{i+\Lambda}(x)=\phi_\Lambda(T_ix)$
for all $i\in\Z^d$ and $\Lambda\in\C$.
(As  usual, when $d=1$, we equate $T$ with $T_1$.)

A probability measure $\mu$ on $(X,\B)$ is a {\em Gibbs  state}
for a specification $\V=\{\phi_\Lambda\}_{\Lambda\in\C}$ if
$$
\mu\bigl([x]_\Lambda\big|\,\B_{\Lambda^c}\bigr)(x)
   = \phi_\Lambda(x)\qquad\hbox{a.e.}(\mu)
$$
for all $\Lambda\in\C$.
The set of Gibbs states is denoted by $\G(\V)$.

If $\V$ is continuous, there exists at least one Gibbs state,
and  natural questions concern the structure of $\G(\V)$,
in particular, uniqueness and any ergodic or mixing properties
of Gibbs states.
In general, $\G(\V)$ is a compact, convex subset of $M(X)$
(the set of probability measures on $(X,\B)$), and
its extreme points are those states for which the tail
$\sigma$-field $\bigcap_{\Lambda\in\C}\B_{\Lambda^c}$
is trivial. If in addition $I=\Z^d$ and $\V$ is $T$-invariant,
then at least one of the Gibbs states is $T$-invariant;
the $T$-invariant Gibbs states form a convex subset of $\G(V)$,
the extreme points of which are the ergodic ones.
Thus, a unique $T$-invariant Gibbs state $\mu$ is necessarily ergodic,
but if it is also the only Gibbs state, then it follows from the
trivially of the tail field that $(T,\mu)$ has some strong mixing properties.
However, from the dynamical systems perspective, the system
with the strongest mixing properties is
one which is isomorphic to, and hence indistinguishable from,
a Bernoulli shift (by which we mean the invertible dynamical system
generated by a sequence of independent identically distributed
random variables taking values in a finite set).
While it is known that for those specifications which are
attractive (or ferromagnetic), uniqueness of the Gibbs state
implies the Bernoulli property (see \cite{A,H1,H2}),
in general this is not the case.

Dobrushin \cite{D} showed that if
$\V=\{\phi_\Lambda\}_{\Lambda\in\C}$ is a continuous specification
such that
$$
\sup_{i\in I}\sum_{s\in S}\sum_{j\neq i}
\sup_{\substack{x_{\{j\}^c}=y_{\{j\}^c}\\x_i=y_i=s}}
|\phi_{\{i\}}(x)-\phi_{\{i\}}(y)|<2,
$$
then $\V$ has a unique Gibbs state $\mu$.
This result holds for all countable $I$ and does not depend
on $I$ having a group structure. If $I=\Z^d$ (in fact, if $I$ is any
amenable group) and $\V$ is $T$-invariant,
then $\mu$ is $T$-invariant and the corresponding
dynamical system $(T,\mu)$ is Bernoulli (see \cite{H2}).

Of particular interest in statistical mechanics
are those specifications determined by an interaction potential
$\Phi=\{\Phi_\Lambda\}_{\Lambda\in\C}$, that is, a set
of functions on $X$ such that $\Phi_\Lambda$
is ${\B}_\Lambda$-measurable.
The corresponding specification $\V_\Phi$ is given by
\begin{equation}\label{eq:pot}
\phi_\Lambda(x)=\exp\bigl(-H^\Phi_\Lambda(x)\bigr)\left\{\sum_{y\in[x]_{\Lambda^c}}
                    \exp\bigl(-H^\Phi_\Lambda(y)\bigr)\right\}^{-1},
\end{equation}
where
\[
H^\Phi_\Lambda(x)=\sum_{\{K\in\C:K\cap\Lambda\neq\emptyset\}}\Phi_K(x).
\]
Usually $\Phi$ is required to satisfy a summability condition
that ensures $H^\Phi_\Lambda$ is continuous,
and hence that $\V_\Phi$ is continuous and positive.
When $I=\Z^d$, the interaction potential $\Phi$ is said to be
$T$-invariant if $\Phi_{\Lambda+i}(x)=\Phi_\Lambda(T_ix)$ for all
$\Lambda\in{\mathcal C}$ and $i\in\Z^d$, in which case
${\mathcal V}_\Phi$ is also $T$-invariant.
An example which has been studied extensively is
the well-known Ising model on $\{-1,1\}^{\Z^d}$,
where $\Phi_\Lambda\equiv0$ for $|\Lambda|\neq2$.
In particular, we note that in dimensions $d\geq2$,
the nearest-neighbour model
(where $\Phi_{\{i,j\}}\equiv0$ unless $|i-j|=1$)
gives examples of non-unique Gibbs states.
When $d=1$, the situation is quite different.
Here, the nearest-neighbour model has a unique Gibbs state;
more generally, if $\Phi_{\{i,j\}}(x)=-J(|i-j|)x_ix_j$,
where $\{J(n)\}$ is a non-negative sequence,
there is a unique Gibbs state provided
$\sum_{i=1}^niJ(i)=o(\sqrt{\log n})$ \cite{RT}.
On the other hand, if $J(n)=\beta n^{-\alpha}$, there is a phase transition
(i.e., non-uniqueness) for large enough $\beta$ when $1<\alpha<2$ \cite{FS}.
(A more thorough survey of such results can be found in \cite{HOG}.)

For potentials in general,
Ruelle \cite{R68} showed that when $I=\Z$ and
$\Phi=\{\Phi_\Lambda\}_{\Lambda\in\C}$ is $T$-invariant,
there is a unique $T$-invariant Gibbs state provided
\begin{equation}\label{eq:h2}
\sum_{\Lambda\ni0}\text{diam}(\Lambda)\|\Phi_\Lambda\|<\infty.
\end{equation}
Gallavotti \cite{G} and Ledrappier \cite{L} independently
proved it has the Bernoulli property.
Coelho and Quas \cite{CQ} pointed out that
by combining Ruelle's characterisation \cite{R67} of the $T$-invariant Gibbs states
of $\Phi$ as equilibrium states of the function
$$
A_\Phi(x)=\sum_{\min\Lambda=0}\phi_\Lambda(x)
$$
with the results of Walters \cite{W},
\eqref{eq:h2}
could be replaced by  the weaker hypothesis
\begin{equation}\label{eq:h3}
\sum_{\min\Lambda=0}\text{diam}(\Lambda)\|\Phi_\Lambda\|<\infty.
\end{equation}

Berbee \cite{B} considered continuous
specifications $\{\phi_\Lambda\}_{\Lambda\in\C}$
determined as follows.
There is a sequence $\{\Lambda_n\}_{n=0}^\infty\subset\C$
such that $\Lambda_n\nearrow I$,
and positive $\B_{\Lambda_{n-1}^c}$-measurable functions $f_n\in C(X)$
$(n\geq0)$ (with $\Lambda_{-1}=\emptyset$)
such that
$$
\phi_{\Lambda_n}(x)
=\left(\prod_{j=0}^nf_j(x)\right)
\left(\sum_{y\in[x]_{\Lambda_n^c}}\prod_{j=0}^nf_j(y)\right)^{-1}
\qquad(n\geq 0,\,x\in X).
$$
Any positive continuous specification can be described in this way
(see Section \ref{sec:spec}).
For a positive function $f$ on $X$ and $\Lambda\subseteq I$, let
$$
r_\Lambda(f)=\sup\{f(x)/f(y)\,:\,x_\Lambda=y_\Lambda\},
\qquad\bar r_\Lambda(f)=\inf\{f(x)/f(y)\,:\,x_\Lambda=y_\Lambda\}.
$$
(Note that $\bar r_\Lambda(f)=r_\Lambda(f)^{-1}$.)
Berbee showed that if
\begin{equation}\label{eq:h4}
\sum_{n=0}^\infty\prod_{k=0}^n\inf_{i\geq0}\bar r_{\Lambda_{i+k}}(f_i)
=\infty,
\end{equation}
then there is a unique Gibbs state.
Like Dobrushin's result, this holds for any countable $I$,
and for non-invariant, as well as invariant, specifications.
Taking $I=\Z$ and
\begin{equation}\label{eq:li}
\Lambda_{2k}=[-k,k],\qquad\Lambda_{2k+1}=[-k,k+1]
\qquad(k\geq0)
\end{equation}
(here and later, we use the usual notation for an interval
$\I\subseteq\R$ to denote $\I\cap\Z$),
and comparing hypotheses for the relevant $f_n$ for an
interaction potential \eqref{eq:ipf+}, it can be seen that Berbee's
hypothesis is weaker than Ruelle's \eqref{eq:h2};
however, while in some cases
it is weaker than that of Coelho and Quas \eqref{eq:h3},
in others it is not.

The hypotheses \eqref{eq:h2}--\eqref{eq:h3} can be viewed as
measures of the continuity of $A(\Phi)$ in terms of its variation
on sites outside finite subsets, or equivalently, in terms of the
rate at which $\|\phi_\Lambda\|$ decreases as $|\Lambda|$ increases.
Similarly, \eqref{eq:h4} is a measure of the continuity
of the functions $f_n$, which in the case of an interaction potential
$\Phi$ with the $f_n$ defined by \eqref{eq:ipf+}, is also related
to the continuity of $A_\Phi$.
In this paper, we consider whether this type of continuity hypothesis
can be relaxed.
We restrict our attention to the case $I=\Z$
since \eqref{eq:h2} and \eqref{eq:h3} are only valid in this case, 
and while Berbee's result holds for any $I$, in the case of
$T$-invariant specifications it would seem that \eqref{eq:h4}
rules out any genuinely multi-dimensional examples.
Indeed, as has already been noted, the nearest-neighbour Ising
model provides examples of non-uniqueness in dimensions greater
than one.

As Dobrushin \cite{D} pointed out, any positive specification
$\V=\{\phi_\Lambda\}_{\Lambda\in\C}$
is completely determined by $\{\phi_{\{n\}}\,:\,n\in\Z\}$;
consequently, if $\V$ is also $T$-invariant, it is determined by $\phi_{\{0\}}$ alone.
Thus, it seems natural to seek conditions for uniqueness and Bernoullicity
in terms of $\phi_{\{0\}}$,
or equivalently, the ratios $\{\phi_{\{0\}}(s_0x)/\phi_{\{0\}}(t_0x)\colon s,t\in S,\,x\in X\}$,
where $s_nx\in X$ $(s\in S,\,x\in X,\,n\in\Z)$ is defined by
\begin{equation*}
(s_nx)_i=
\begin{cases}
s,&i=n,\\
x_i,&i\neq n.
\end{cases}
\end{equation*}
We show the following:

\pagebreak

\begin{theorem}\label{thm:bern}
Let $\V=\{\phi_\Lambda\}_{\Lambda\in\C}$
be a positive, continuous, $T$-invariant specification on $S^\Z$,
and let
$$
f(x)=\phi_{\{0\}}(x)/\phi_{\{0\}}(s_0x)\qquad(x\in X)
$$
for some $s\in S$.
If, for some $0<\alpha\leq1$,
\begin{equation*}
\prod_{i=0}^nr_{(-\infty,i]}(f)
=O\bigl(n^{1-\alpha}\bigr),
\end{equation*}
and
\begin{equation*}
\sum_{i=2^n}^{2^{n+1}}
\Bigl(\log r_{[-i,\infty)}(f)\Bigr)^{2\alpha}=o(1),
\end{equation*}
then $\V$ has a unique Gibbs state $\mu$, which is $T$-invariant,
and $(T,\mu)$ is Bernoulli.
\end{theorem}

In particular, there is a unique, Bernoulli Gibbs state if
$$
\max\left\{\log r_{(-\infty,n]}(f),\,\log r_{[-n,\infty)}(f)\right\}
\leq\beta n^{-1}+a_n
\qquad(n\geq1)
$$
for some $\beta<1/2$ and non-negative series $\sum a_n<\infty$.
In fact, we prove Theorem \ref{thm:bern} under slightly weaker hypotheses
(Theorem \ref{thm:bern1}),
which allows the following version to be deduced when the specification is given
in terms of an interaction potential.

\begin{theorem}\label{thm:bern2}
Let $\Phi=\{\Phi_\Lambda\}_{\Lambda\in\C}$
be a $T$-invariant interaction potential on $S^\Z$. If
\begin{equation}\label{hyp3}
\sum_{\substack{\Lambda\ni0\\\Lambda\cap[n,\infty)\neq\emptyset}}
\var(\Phi_\Lambda)
\leq\beta n^{-1}+a_n
\qquad(n\geq1)
\end{equation}
for some $\beta<1/2$ and $a_n\geq0$ such that $\sum a_n<\infty$,
then $\Phi$ has a unique Gibbs state $\mu$, which is $T$-invariant,
and $(T,\mu)$ is Bernoulli.
\end{theorem}

This includes all interaction potentials satisfying \eqref{eq:h2}.
In some cases,
such as when $\sup\{|\Lambda| \,:\,\Phi_\Lambda\not\equiv0\}<\infty$,
\eqref{hyp3} is also weaker than \eqref{eq:h3},
but in others it is not.
Comparison with \eqref{eq:h4} is less straightforward,
since \eqref{eq:h4} applies to non-invariant specifications
and arbitrary increasing sequences $\{\Lambda_n\}_{n=0}^\infty\subset\C$.
Consideration of the Ising model suggests Theorems \ref{thm:bern} and \ref{thm:bern2}
can give better results for uniqueness in some cases,
but the main difference is that the Bernoulli property is obtained.

The idea of the proof of Theorem \ref{thm:bern1} is to characterise a Gibbs state $\mu$
in terms of its conditional probabilities
$$
g(x)=\mu\bigl([x]_{\{0\}}\big|\,\B_{(-\infty,-1]}\bigr)(x)\qquad(x\in X),
$$
in other words, as a $g$-chain; uniqueness, $T$-invariance and the Bernoulli property
can then be determined from similar results for $g$-chains.
Here we use a result (Theorem \ref{thm:c}) which is a formulation
based on the techniques of Johansson, \"Oberg and Pollicott \cite{JOP}.
The estimate of the variation of $\log g$ needed for Theorem \ref{thm:c} is 
provided by Theorem \ref{thm:g}.
This result may be of some interest in its own right
as it gives a sufficient condition for the set of Gibbs states of a $T$-invariant specification to be
the set of $g$-chains for some continuous function $g$.
(A characterisation of $g$-chains which are also Gibbs states,
as well as a discussion of the problem of determining whether a Gibbs state is a $g$-chain,
can be found in \cite{BFV}.)
Theorem \ref{thm:g} follows from Proposition \ref{prop:g},
which gives a general criterion for
the set of Gibbs states to be the same as
the set of $g$-chains for some continuous function $g$,
and Proposition \ref{prop:rb},
which is an adaptation of Berbee's method of ratio bounds in \cite{B}.

Lastly, we consider some alternative hypotheses which guarantee a unique $T$-invariant
Gibbs state but not the Bernoulli property.

\section{Specifications and Gibbs states}\label{sec:spec}

We start with some preliminary remarks about specifications and
Gibbs states; a more detailed account can be found in \cite{HOG}, for example.
The following notation will be used throughout:
if $x\in S^{\Lambda_1}$, $y\in S^{\Lambda_2}$ are sequences
and $\Lambda_1\cap\Lambda_2=\emptyset$,
then $xy\in S^{\Lambda_1\cup\Lambda_2}$ denotes the sequence such that
\begin{equation*}
(xy)_i=
\begin{cases}
x_i,&i\in\Lambda_1,\\
y_i,&i\in\Lambda_2.
\end{cases}
\end{equation*}
The set of bounded Borel functions on $X$ is denoted by $B(X)$.

Let $X=S^\Z$ be a sequence space as described in the introduction,
and let $\V=\{\phi_\Lambda\}_{\Lambda\in\C}$ be a specification on $X$.
Then $\V$ determines a set of probability measures
$\pi^\V_\Lambda(\,\cdot\,|\,x)$ $(\Lambda\in\C,\ x\in X)$
as follows:
$$
\pi^\V_\Lambda(f\,|\,x)
=\sum_{y\in[x]_\Lambda}\phi_\Lambda(y)f(y)
\qquad(f\in C(X)),
$$
and $\mu\in\G(\V)$ if and only if
$$
\mu\bigl([x]_\Lambda\big|\,\B_{\Lambda^c}\bigr)(x)
=\pi^\V_\Lambda\bigl([x]_\Lambda\big|\,x\bigr)\quad{\rm a.e.}\,(\mu)
\qquad (\Lambda\in\C).
$$Thus, if $\V$ is continuous,
$\G(\V)$ is closed (in the weak-* topology),
and moreover,
any weak-* limit of measures of the form
$$
\pi^\V_{\Lambda_n}(\,\cdot\,|\,x^{(n)})
\qquad(x^{(n)}\in X,\ \Lambda_n\nearrow \Z)
$$
is a Gibbs state for $\V$.
Since $M(X)$ is compact, it follows that
there is at least one Gibbs state $\mu$,
and the uniqueness of $\mu$ is equivalent to
the uniform convergence (for each $f\in C(X)$)
of $\pi^\V_\Lambda(f\,|\,x)$ to $\mu(f)$
as $\Lambda\nearrow\Z$.

Similarly, if $\V$ is $T$-invariant,
any weak-* limit of measures of the form
$$
\mu_n(f)=n^{-1}\sum_{i=0}^{n-1}\pi^\V_{[0,n-1]}(fT^i\,|\,x^{(n)})
\qquad(f\in C(X),\,x^{(n)}\in X)
$$
is a $T$-invariant Gibbs state for $\V$.
By compactness, there is always one such measure $\mu$,
and its uniqueness is equivalent to
$$
\lim_{n\to\infty}n^{-1}\sum_{i=0}^{n-1}\pi^\V_{[0,n-1]}(fT^i\,|\,x)=\mu(f)
$$
uniformly in $x$ for all $f\in C(X)$.

Note that if $\phi_\Lambda$ is given for some $\Lambda\in\C$,
then (\ref{1aii}) determines $\phi_K$ for all $K\subseteq\Lambda$.
Moreover, if (\ref{1aii}) holds for the pairs $K,\Lambda$
and $K',K$, where $K'\subseteq K\subseteq\Lambda$,
then it holds for $K',\Lambda$.
Thus, to define a specification,
it is enough to consider a sequence
$\{\Lambda_n\}_{n=0}^\infty\subset\C$
such that $\Lambda_n\nearrow \Z$ and specify the $\phi_{\Lambda_n}$
so that (\ref{1ai}) and (\ref{1aii}) hold for
$\Lambda_{n-1},\Lambda_n$ $(n\geq 1)$.

Let $\V$ be positive;
note that for positive specifications, \eqref{1aii} is equivalent to
\begin{equation}\label{1aiii}
\frac{\phi_K(x)}
{\phi_K(x^*_Kx^{\phantom{*}}_{K^c})}
=
\frac{\phi_\Lambda(x)}
{\phi_\Lambda(x^*_Kx^{\phantom{*}}_{K^c})}
\qquad(K\subseteq\Lambda,\ x,x^*\in X).
\end{equation}
If $\Lambda\in\C$
and $\{i_0,\ldots,i_n\}$ is an enumeration of $\Lambda$,
with $\Lambda_j=\{i_0,\ldots,i_j\}$ ($0\leq j\leq n$),
then for any $x,x^*\in X$, we have
\begin{align*}
\frac{\phi_{\Lambda}(x)}
{\phi_\Lambda(x_\Lambda^*x_{\Lambda^c}^{\phantom{*}})}
&=
\frac{\phi_{\Lambda}(x)}
{\phi_{\Lambda}(x_{\Lambda_{n-1}}^*x_{\Lambda_{n-1}^c}^{\phantom{*}})}
\,\cdot\,\frac
{\phi_{\Lambda}(x_{\Lambda_{n-1}}^*x_{\Lambda_{n-1}^c}^{\phantom{*}})}
{\phi_{\Lambda}(x_{\Lambda}^*x_{\Lambda^c}^{\phantom{*}})}\\
&=\frac{\phi_{\Lambda_{n-1}}(x)}
{\phi_{\Lambda_{n-1}}(x_{\Lambda_{n-1}}^*
x_{\Lambda_{n-1}^c}^{\phantom{*}})}
\,\cdot\,\frac
{\phi_{\{i_n\}}(x_{\Lambda_{n-1}}^*x_{\Lambda_{n-1}^c}^{\phantom{*}})}
  {\phi_{\{i_n\}}(x_{\Lambda}^*x_{\Lambda^c}^{\phantom{*}})}
&\qquad(\text{by \eqref{1aiii}}),
\end{align*}
and so it follows that
\begin{equation}\label{eq:spec}
\frac{\phi_{\Lambda}(x)}
{\phi_{\Lambda}(x_{\Lambda}^*x_{\Lambda^c}^{\phantom{*}})}
=\prod_{j=0}^n\frac
{\phi_{\{i_j\}}(x_{\Lambda_{j-1}}^*x_{\Lambda_{j-1}^c}^{\phantom{*}})}
{\phi_{\{i_j\}}(x_{\Lambda_j}^*x_{\Lambda_j^c}^{\phantom{*}})}
\qquad(x,x^*\in X,\ n\geq 0).
\end{equation}
(Here and later, $\Lambda_{-1}$ is taken to be $\emptyset$.)
This, together with \eqref{1ai}, determines $\phi_\Lambda$.
In particular, any positive specification
$\V=\{\phi_\Lambda\}_{\Lambda\in\C}$
is completely determined by the $\phi_{\{n\}}$ $(n\in\Z)$.
Moreover, $\V$ is $T$-invariant provided
$\phi_{\{n\}}=\phi_{\{0\}}\circ T^n$ for all $n\in\Z$,
in which case it is determined solely by $\phi_{\{0\}}$.

Conversely, given a sequence of positive functions $\{\phi_{\{n\}}\}_{n\in\Z}\subset B(X)$
such that $\sum_{y\in[x]_{\{n\}}}\phi_{\{n\}}(y)=1$ for all $x\in X$,
and $\{\Lambda_n\}_{n=0}^\infty\subset\C$ such that $\Lambda_n\nearrow\Z$
and $|\Lambda_n\setminus\Lambda_{n-1}|=1$ $(n\geq0)$,
define $\phi_{\Lambda_n}$ by \eqref{1ai} and
\begin{equation*}
\frac{\phi_{\Lambda_n}(x)}
{\phi_{\Lambda_n}(x_{\Lambda_n}^*x_{\Lambda_n^c}^{\phantom{*}})}
=\prod_{j=0}^n\frac
{\phi_{\{i_j\}}(x_{\Lambda_{j-1}}^*x_{\Lambda_{j-1}^c}^{\phantom{*}})}
{\phi_{\{i_j\}}(x_{\Lambda_j}^*x_{\Lambda_j^c}^{\phantom{*}})}
\qquad(x,x^*\in X,\ n\geq1),
\end{equation*}
where $\{i_n\}=\Lambda_n\setminus\Lambda_{n-1}$.
Clearly, the $\phi_{\Lambda_n}$ satisfy \eqref{1aiii},
and so determine a specification, which is independent of the choice of $\{\Lambda_n\}$
since a specification is determined solely by the $\phi_{\{n\}}$.
Thus, any sequence of positive functions $\{\phi_{\{n\}}\}_{n\in\Z}\subset B(X)$
such that $\sum_{y\in[x]_{\{n\}}}\phi_{\{n\}}(y)=1$ for all $x\in X$
determines a unique positive specification.
In particular, any positive $f\in B(X)$ determines a (unique) $T$-invariant specification
$\V=\{\phi_\Lambda\}_{\Lambda\in\C}$ such that
\begin{equation*}
\phi_{\{n\}}(x)=\frac{f(T^nx)}{\sum_{y\in[T^nx]_{\{0\}}}f(y)}
\qquad(x\in X,\,n\in\Z).
\end{equation*}

By a {\it specification sequence\/},
we mean a sequence $\{\Lambda_n,f_n\}_{n=0}^\infty\subset\C\times B(X)$
such that $\Lambda_{n-1}\subset\Lambda_n$,
$|\Lambda_n\setminus\Lambda_{n-1}|=1$,
and $f_n$ is positive and $\B_{\Lambda_{n-1}^c}$-measurable;
it is a specification sequence for the specification $\V=\{\phi_\Lambda\}_{\Lambda\in\C}$ if
\begin{equation}\label{eq:fdef}
\phi_{\Lambda_n}(x)
=\left(\prod_{j=0}^nf_j(x)\right)
\left(\sum_{y\in[x]_{\Lambda_n^c}}\prod_{j=0}^nf_j(y)\right)^{-1}
\qquad(n\geq 0,\,x\in X).
\end{equation}
If $\V$ is positive, such a sequence  exists for any
choice of $\{\Lambda_n\}$.
Indeed, let $x^*\in X$ and define
\begin{equation}\label{eq:fbar}
f_n(x)=\frac{\phi_{\{i_n\}}
(x_{\Lambda_{n-1}}^*x_{\Lambda_{n-1}^c}^{\phantom{*}})}
{\phi_{\{i_n\}}(x_{\Lambda_n}^*x_{\Lambda_n^c}^{\phantom{*}})}
\qquad(x\in X,\ n\geq 0),
\end{equation}
where $\Lambda_n\setminus\Lambda_{n-1}=\{i_n\}$.
Then $f_n$ is suitably measurable
and it follows from \eqref{eq:spec} that \eqref{eq:fdef} holds.
If $\V$ is continuous, so are the $f_n$,
and if $\V$ is also $T$-invariant, then a consequence of
\eqref{eq:fbar} is that if
\begin{equation}\label{eq:fs}
f(x)=\phi_{\{0\}}(x)/\phi_{\{0\}}(s_0x)\qquad(x\in X)
\end{equation}
for some $s\in S$, then a specification sequence
$\{[0,n],f_n\}_{n=0}^\infty$
can be chosen so that $f_0=f$ and 
\begin{equation}\label{eq:f+var}
r_{[-n,i+j]}(f_j)=r_{[-n-j,i]}(f_0)\leq r_{[-n,i]}(f)
\qquad(i,j,n\geq0).
\end{equation}
Conversely, any specification sequence $\{\Lambda_n,f_n\}_{n=0}^\infty$
such that $\Lambda_n\nearrow \Z$ determines a (unique) specification
$\V=\{\phi_\Lambda\}_{\Lambda\in\C}$ via \eqref{eq:fdef},
since \eqref{1aiii} holds for the pairs $\Lambda_{n-1}$, $\Lambda_n$.

Any positive, continuous, $T$-invariant specification $\V$ can be
determined by a $T$-invariant interaction potential $\Phi=\{\Phi_\Lambda\}_{\Lambda\in\C}$
such that for each $\Lambda\in\C$,
$$
\sum_{\Lambda^\prime\cap\Lambda\neq\emptyset}
\bigl|\Phi_{\Lambda^\prime}(x)-\Phi_\Lambda(x)\bigr|
$$
converges uniformly in $x$ (see \cite{S} for details).
If $\sum_{\Lambda\ni0}\|\Phi_\Lambda\|<\infty$,
the functions
\begin{equation}\label{eq:ipf+}
f_n(x)=\exp\left(
-\sum_{\substack{\Lambda\ni n\\\Lambda\cap[0,n-1]=\emptyset}}
\Phi_\Lambda(x)\right)\qquad(n\geq 0,\ x\in X)
\end{equation}
determine a specification sequence $\{[0,n],f_n\}_{n=0}^\infty$
for $\V$,
and since $\Phi$ is $T$-invariant,
\begin{equation}\label{eq:ipr}
\max\bigl\{\log r_{(-\infty,i+n]}(f_i),\,\log r_{[-n,\infty)}(f_i)\bigr\}
\leq \sum_{\substack{\Lambda\cap[n+1,\infty)\neq\emptyset\\0\in\Lambda}}
\var(\Phi_\Lambda)
\qquad(i,n\geq0).
\end{equation}

\section{A ratio bound}\label{sec:rb}

Given a specification sequence $\bigl\{[0,n],f_n\bigr\}_{n=0}^\infty$,
define $L_{m,n}\colon B(X)\rightarrow B(X)$ $(0\leq m\leq n)$ by
\begin{equation*}
L_{m,n}f(x)=
\sum_{y\in[x]_{[m,n]^c}}\prod_{i=m}^{n}f_i(y)f(y)
\qquad(f\in B(X)),
\end{equation*}
and let
\begin{gather*}
\rho_k^{(n)}(x,y)=\inf\left\{
\frac{L_{m,m+n}\1[\zeta](x)}{L_{m,m+n}\1[\zeta](y)}
\,:\,\zeta\in S^{[m,m+k]},\ m\geq 0\right\},\\
\noalign{\vskip 10pt}
\rho_{-1}^{(n)}(x,y)=\inf\left\{
\frac{L_{m,m+n}1(x)}{L_{m,m+n}1(y)}
\,:\,m\geq 0\right\}
\qquad(k,n\geq 0,\,x,y\in X).
\end{gather*}
In addition to the usual notation for sets of integers,
we denote the sets of negative and non-positive integers
by $\N^-$ and $\Z^-$ respectively.

\begin{proposition}\label{prop:rb}
Let $\bigl\{[0,n],f_n\bigr\}_{n=0}^\infty$ be a specification sequence,
 $A\subseteq\N^-$, and
\begin{equation*}
v_k=\inf_{i\geq0}\bar r_{A\cup[0,i+k]}(f_i)
\qquad(k\geq 0).
\end{equation*}
Then
$$
\liminf_{n\to\infty}\ 
\inf_{x_A=y_A}\rho_0^{(n)}(x,y)\rho_0^{(n)}(y,x)
\geq\lim_{n\to\infty}\Biggl(\sum_{k=0}^n\prod_{j=0}^kv_j\Biggr)^2
\Biggl(1+\sum_{k=0}^n\prod_{j=0}^kv_j\Biggr)^{-2}.
$$
\end{proposition}

We split the proof of Proposition \ref{prop:rb} into two lemmas.
The proof of the first one is based on that of  \cite[Proposition 2.1]{B}.

\begin{lemma}\label{lem:rb1}
Let $\{f_n\}_{n=0}^\infty\subset B(X)$ and $A$ be as above,
and let $x,y\in X$ be such that $x_A=y_A$.
If $\{v_k\}_{k=0}^\infty$ is an increasing, positive sequence such that
\begin{equation*}
v_k\leq\inf_{i\geq0}\bar r_{A\cup[0,i+k]}(f_i)
\qquad(k\geq 0).
\end{equation*}
then
\begin{equation*}
\rho_k^{(n)}(x,y)
\geq v_k\rho_{k-1}^{(n-1)}(x,y)
+\sum_{j=k}^{n-1}(v_{j+1}-v_j)\rho_j^{(n-1)}(x,y)
\qquad(0\leq k\leq n).
\end{equation*}
\end{lemma}

\begin{proof}
We denote $\rho_k^{(n)}(x,y)$ by $\rho_k^{(n)}$,
and $[\zeta]_{[i,j]}$ by $[\zeta]_i^j$
($\zeta\in S^\Lambda$, $[i,j]\subseteq\Lambda\subseteq\Z$).
Let
\begin{equation*}
f^{(k)}_m(w)
=\inf\bigl\{f_m(z)\,:\,z\in[x]_A\cap[w]_0^{m+k}\bigr\},
\end{equation*}
and
\begin{equation*}
\phi^{(k+1)}_m(w)=f^{(k+1)}_m(w)-f^{(k)}_m(w)
\qquad(k,m\geq 0,\ w\in X).
\end{equation*}
Note that
\begin{equation*}
f^{(k)}_m(w)\geq v_kf_m(w)\qquad(w\in[x]_A),
\end{equation*}
and
$$
f_m\geq f_m^{(k)}
+\sum_{j=k+1}^N\phi_m^{(j)}\qquad(N>k).
$$
Let $m\geq0$, $n\geq1$, and $\zeta\in S^{[m,m+k]}$ for some  $0\leq k\leq n$.
Since $f^{(j)}_m(w)$, $\phi_{m}^{(j)}(w)$ depend only on $w_{[m,m+j]}$,
we can identify them with the functions induced by
the natural projection onto $S^{[m,m+j]}$.
Then
\begin{align*}
L_{m+1,m+n}
\bigl(f^{(k)}_m\1[\zeta]\bigr)(\zeta_{\{m\}}x_{\{m\}^c})
&=f^{(k)}_m(\zeta)
L_{m+1,m+n}\bigl(\1[\zeta]_{m+1}^{m+k}\bigr)(x)\\
&\geq
\rho_{k-1}^{(n-1)}f^{(k)}_m(\zeta)
L_{m+1,m+n}\bigl(\1[\zeta]_{m+1}^{m+k}\bigr)(y)\\
&=
\rho_{k-1}^{(n-1)}
L_{m+1,m+n}\bigl(f^{(k)}_m\1[\zeta]\bigr)
(\zeta_{\{m\}}y_{\{m\}^c})\\
&\hskip 0.75 in\qquad(0\leq k\leq n,\,n\geq1,\,m\geq 0),
\end{align*}
and
\begin{multline*}
L_{m+1,m+n}
\bigl(\phi_{m}^{(j)}\1[\zeta]\bigr)(\zeta_{\{m\}}x_{\{m\}^c})\\
\shoveright{\begin{aligned}
&=\sum_{\eta\in S^{[m+k+1,m+j]}}\phi_{m}^{(j)}(\zeta\eta)
L_{m+1,m+n}\bigl(\1[\zeta\eta]_{m+1}^{m+j}\bigr)(x)\\
&\geq
\rho_{j-1}^{(n-1)}\sum_{\eta\in S^{[m+k+1,m+j]}}\phi_{m}^{(j)}(\zeta\eta)
L_{m+1,m+n}\bigl(\1[\zeta\eta]_{m+1}^{m+j}\bigr)(y)\\
&=\rho_{j-1}^{(n-1)}
L_{m+1,m+n}\bigl(\phi_{m}^{(j)}\1[\zeta]\bigr)
(\zeta_{\{m\}}y_{\{m\}^c})
\end{aligned}}\\
(1\leq k+1\leq j\leq n,\,m\geq 0).
\end{multline*}
Therefore,
\begin{align*}
L_{m,m+n}\1[\zeta](x)
&\geq
L_{m+1,m+n}
\bigl(f_{m}^{(k)}\1[\zeta]\bigr)
(\zeta_{\{m\}}x_{\{m\}^c})\\
&\qquad+\sum_{j=k+1}^n
L_{m+1,m+n}\bigl(\phi_{m}^{(j)}\1[\zeta]\bigr)
(\zeta_{\{m\}}x_{\{m\}^c})\\
&\geq
\rho_{k-1}^{(n-1)}L_{m+1,m+n}
\bigl(f_{m}^{(k)}\1[\zeta]\bigr)
(\zeta_{\{m\}}y_{\{m\}^c})\\
&\qquad+\sum_{j=k+1}^n\rho_{j-1}^{(n-1)}
L_{m+1,m+n}\bigl(\phi_{m}^{(j)}\1[\zeta]\bigr)
(\zeta_{\{m\}}y_{\{m\}^c})\\
&=
\sum_{j=k}^{n-1}\bigl(\rho_{j-1}^{(n-1)}-\rho_j^{(n-1)}\bigr)L_{m+1,m+n}
\bigl(f_{m}^{(j)}\1[\zeta]\bigr)
(\zeta_{\{m\}}y_{\{m\}^c})\\
&\qquad+\rho^{(n-1)}_{n-1}L_{m+1,m+n}\bigl(f_{m}^{(n)}\1[\zeta]\bigr)
(\zeta_{\{m\}}y_{\{m\}^c})\\
&\geq
\sum_{j=k}^{n-1}\bigl(\rho_{j-1}^{(n-1)}-\rho_j^{(n-1)}\bigr)v_j
L_{m+1,m+n}\bigl(f_{m}\1[\zeta]\bigr)
(\zeta_{\{m\}}y_{\{m\}^c})\\
&\qquad+\rho^{(n-1)}_{n-1}v_n
L_{m+1,m+n}\bigl(f_m\1[\zeta]\bigr)(\zeta_{\{m\}}y_{\{m\}^c})\\
&=
\left(\rho^{(n-1)}_{n-1}v_n
+\sum_{j=k}^{n-1}\bigl(\rho_{j-1}^{(n-1)}-\rho_j^{(n-1)}\bigr)v_j\right)
L_{m,m+n}\1[\zeta](y)
\end{align*}
for $0\leq k\leq n-1,\,m\geq 0$,
and
\begin{equation*}
L_{m,m+n}\1[\zeta](x)
=\rho^{(n-1)}_{n-1}v_nL_{m,m+n}\1[\zeta](y)
\end{equation*}
for $k=n\geq1$.
Hence,
\begin{equation*}
\rho_k^{(n)}
\geq \rho^{(n-1)}_{n-1}v_n
+\sum_{j=k}^{n-1}\bigl(\rho_{j-1}^{(n-1)}-\rho_j^{(n-1)}\bigr)v_j
\qquad(0\leq k\leq n,\,n\geq1)
\end{equation*}
(where the empty sum is taken to be 0).
Rearranging, we have
\begin{equation*}
\rho_k^{(n)}
\geq\rho_{k-1}^{(n-1)}v_k+\sum_{j=k}^{n-1}\rho_j^{(n-1)}(v_{j+1}-v_j)
\qquad(0\leq k\leq n,\,n\geq1),
\end{equation*}
as required.
\end{proof}

Now, given an increasing sequence $\{v_k\}_{k=0}^\infty$
such that $0<v_k\leq 1$ for all $k$,
consider $p_k^{(n)}$ $(-1\leq k\leq n,\,n\geq0)$,
defined recursively by
\begin{align}
p_{-1}^{(n)}&=1,\label{eq:rbr0}\\
p_0^{(0)}&=v_0,\label{eq:rbr1}\\
p^{(n)}_k&=v_kp_{k-1}^{(n-1)}+\sum_{j=k}^{n-1}(v_{j+1}-v_j)p_j^{(n-1)}
\qquad(0\leq k\leq n,\,n\geq1).\label{eq:rbr2}
\end{align}
(As before, when $k=n$, the empty sum is taken to be 0.)

\begin{lemma}\label{lem:rb2}
Let $p_k^{(n)}$ $(k\geq-1,\,n\geq0)$ and $\{v_k\}_{k=0}^\infty$
be as above. Then
\begin{equation*}
\lim_{n\to\infty}p^{(n)}_0
\geq\lim_{n\to\infty}\Biggl(\sum_{k=0}^n\prod_{j=0}^kv_j\Biggr)
\Biggl(1+\sum_{k=0}^n\prod_{j=0}^kv_j\Biggr)^{-1}.
\end{equation*}
\end{lemma}

\begin{proof}
It follows from \eqref{eq:rbr2} that for $k\leq n$,
$$
p^{(n+1)}_k-p^{(n)}_k
=v_k(p_{k-1}^{(n)}-p_{k-1}^{(n-1)})
+(v_{n+1}-v_n)p_n^{(n)}
 +\sum_{j=k}^{n-1}(v_{j+1}-v_j)(p_j^{(n)}-p_j^{(n-1)}),
$$
and from \eqref{eq:rbr0}--\eqref{eq:rbr1} that
\begin{equation*}
p_0^{(1)}-p_0^{(0)}=(v_1-v_0)v_0\geq0.
\end{equation*}
Hence, it follows inductively that
$p^{(n+1)}_n\geq p^{(n)}_n$ for all $n\geq0$,
and then that
$p_k^{(n+1)}\geq p_k^{(n)}$ for all $0\leq k\leq n$.
Thus
$$
p_k=\lim_{n\to\infty}p_k^{(n)}\qquad(k\geq -1)
$$
is well defined, and
\begin{equation*}
p_k=v_kp_{k-1}+\sum_{j=k}^\infty(v_{j+1}-v_j)p_j
\qquad(k\geq 0).
\end{equation*}
Therefore,
\begin{align*}
p_k-p_{k+1}
&=v_kp_{k-1}-v_{k+1}p_k+(v_{k+1}-v_k)p_k\\
&=v_k(p_{k-1}-p_k),
\end{align*}
and so,
$$
p_k-p_{k+1}=(1-p_0)\prod_{j=0}^kv_j.
$$
Summing gives
\begin{equation*}
p_0-p_{n+1}=(1-p_0)\sum_{k=0}^n\prod_{j=0}^kv_j.
\end{equation*}
Since $p_{n+1}\geq0$,
\begin{equation*}
p_0\geq\Biggl(\sum_{k=0}^n\prod_{j=0}^kv_j\Biggr)
\Biggl(1+\sum_{k=0}^n\prod_{j=0}^kv_j\Biggr)^{-1},
\end{equation*}
from which the result follows.
\end{proof}

\begin{proof}[Proof of Proposition \ref{prop:rb}]
It follows from Lemma \ref{lem:rb1}
and the Cauchy-Schwartz inequality that
\begin{multline}\label{eq:rbcs}
\bigl(\rho_k^{(n)}(x,y)\rho_k^{(n)}(y,x)\bigr)^{\frac{1}{2}}\\
\geq
v_k\bigl(\rho_{k-1}^{(n-1)}(x,y)\rho_{k-1}^{(n-1)}(y,x)\bigr)
^{\frac{1}{2}}+\sum_{j=k}^{n-1}(v_{j+1}-v_j)
\bigl(\rho_j^{(n-1)}(x,y)\rho_j^{(n-1)}(y,x)\bigr)^{\frac{1}{2}}\\
\qquad(0\leq k\leq n,\,n\geq1)
\end{multline}
for any $x,y\in X$ such that $x_A=y_A$.
Note also that $\rho_{-1}^{(n)}(x,y)\rho_{-1}^{(n)}(y,x)=1$
for all $n\geq 0$ and $x,y\in X$.
Define $p^{(n)}_k$ $(-1\leq k\leq n,\,n\geq 0)$ inductively
as in \eqref{eq:rbr0}--\eqref{eq:rbr2} with
$\{v_k\}$ as in the statement of Proposition \ref{prop:rb}.
Then it follows from \eqref{eq:rbr0}--\eqref{eq:rbcs} that
\begin{equation*}
0\leq p_k^{(n)}
\leq\inf_{x_A=y_A}
\bigl(\rho_k^{(n)}(x,y)\rho_k^{(n)}(y,x)\bigr)^{\frac{1}{2}}
\leq 1\qquad(-1\leq k\leq n,\,n\geq 0).
\end{equation*}
The result now follows from Lemma \ref{lem:rb2}.
\end{proof}

\section{Gibbs states and {\rm g}-chains}

Let
$$
\G_X=\biggl\{g\in B(X)\,:\,g\text{ is $\B_{\Z^-}$-measurable},
\ g\geq0,\ \sum_{s\in S}g(s_0x)=1\ \forall x\in X\biggr\},
$$
and let $g\in\G_X$. A measure $\mu\in M(X)$ is a $g$-chain if
$$
\mu\bigl([x]_{\{n\}}\big|\,\B_{(-\infty,n-1]}\bigr)(x)
=g(T^nx)\qquad\hbox{a.e.}(\mu)
$$
for all $n\in\Z$.
A $T$-invariant $g$-chain is usually referred to as a $g$-measure.
If $g\in\G_X$ is continuous, there is always at least
one $g$-measure.

\begin{proposition}\label{prop:g}
Let $\V=\{\phi_\Lambda\}_{\Lambda\in\C}$
be a positive, continuous, $T$-invariant specification.
If $\pi^\V_{[0,n]}\bigl(\,[x]_{\{0\}}\big|\,x\bigr)$
converges to a continuous limit in $x\in X$ as $n\to\infty$, then
$$
g_\V(x)=\lim_{n\to\infty}\pi^\V_{[0,n]}\bigl(\,[x]_{\{0\}}\big|\,x\bigr)
\qquad(x\in X)
$$
is a positive, continuous function in $\G_X$,
and $\mu\in\G(\V)$ if and only if $\mu$ is a $g_\V$-chain.
\end{proposition}

\begin{proof}
Let
\begin{equation}\label{eq:propg1}
g_\V(x)=\lim_{n\to\infty}\pi^\V_{[0,n]}\bigl(\,[x]_{\{0\}}\big|\,x\bigr)
\qquad(x\in X)
\end{equation}
as above; by hypothesis, $g_\V$ is continuous.
Note that $g_\V(x)$ is $\B_{\Z^-}$-measurable,
and since
$$
\inf\phi_{\{0\}}
\leq\pi^\V_{[0,n]}\bigl(\,[x]_{\{0\}}\big|\,x\bigr)
\leq\sup\phi_{\{0\}}
\qquad(n\geq0,\,x\in X),
$$
it follows that  $g_\V$ is positive.
Clearly, $\sum_{s\in S}g_\V(s_0x) =1$ for all $x\in X$,
and so $g_\V\in\G_X$.
Moreover, it follows from the Martingale Theorem that
$$
\mu\bigl(\,[x]_{\{n\}}\big|\,\B_{(-\infty,n-1]}\bigr)(x)=g_\V(T^nx)
\qquad\text{a.e.}\,(\mu)
$$
for any Gibbs state $\mu$ of $\V$ when $n=0$,
and hence, by the $T$-invariance of $\V$, for all $n\in\Z$;
thus, any Gibbs state is a $g_\V$-chain.

Conversely, if $\mu$ is a $g_\V$-chain, then
\begin{align*}
\mu\bigl(\,[x]_{\{n\}}\big|\,\B_{\{n\}^c}\bigr)(x)
&=\lim_{m\to\infty}\mu\bigl(\,[x]_{\{n\}}\big|\,\B_{(-\infty,m]\setminus\{n\}}\bigr)(x)\\
&=\lim_{m\to\infty}\frac{\prod_{i=n}^mg_\V(T^ix)}
{\sum_{s\in S}\prod_{i=n}^mg_\V(T^i(s_nx))}
\qquad\text{a.e.}\,(\mu)
\quad(n\in\Z).
\end{align*}
Given $n\in\Z$, consider $k>m>n$.
From \eqref{eq:propg1} we have
\begin{align*}
\prod_{i=n}^mg_\V(T^ix)
&=\lim_{k\to\infty}\prod_{i=n}^m\pi^\V_{[i,k]}\bigl([x]_{\{i\}}\big|\,x\bigr)\\
&=\lim_{k\to\infty}\pi^\V_{[n,k]}\bigl([x]_{[n,m]}\big|\,x\bigr)\\
&=\lim_{k\to\infty}\sum_{w\in[x]_{[m+1,k]^c}}
\pi^\V_{\{n\}}\bigl([x]_{\{n\}}\big|\,w\bigr)\,
\pi^\V_{[n+1,k]}\bigl([w]_{[n+1,k]}\big|\,x\bigr)
\end{align*} for all $x\in X$.
Therefore, for all $s,t\in S$, $x\in X$,
\begin{align*}
\frac{\mu\bigl([s]_n\,\big|\,\B_{\{n\}^c}\bigr)(x)}
{\mu\bigl([t]_n\,\big|\,\B_{\{n\}^c}\bigr)(x)}
&\leq\lim_{m\to\infty}\lim_{k\to\infty}\sup_{w\in[x]_{[m+1,k]^c}}
\frac{\pi^\V_{\{n\}}\bigl([s]_n\,\big|\,w\bigr)}
{\pi^\V_{\{n\}}\bigl([t]_n\,\big|\,w\bigr)}\\
&=\frac{\phi_{\{n\}}(s_nx)}{\phi_{\{n\}}(t_nx)}
\qquad\text{a.e.}\,(\mu).
\end{align*}
It follows that
$\mu\bigl([s]_n\big|\B_{\{n\}^c}\bigr)(x)=\phi_{\{n\}}(s_nx)$
a.e.$(\mu)$ for all $s\in S$, $n\in\Z$,
and so $\mu\in\G(\V)$.
\end{proof}

\begin{theorem}\label{thm:g}
Let $\V=\{\phi_\Lambda\}_{\Lambda\in\C}$
be a positive, continuous, $T$-invariant specification
with specification sequence $\{[0,n],f_n\}_{n=0}^\infty$. If
$$
\sum_{k=0}^\infty\prod_{j=0}^k\inf_{i\geq0}\bar r_{(-\infty,i+j]}(f_i)
=\infty,
$$
then
$$
g_\V(x)=\lim_{n\to\infty}\pi^\V_{[0,n]}\bigl(\,[x]_{\{0\}}\,\big|\,x\bigr)
$$
is well defined, the convergence is uniform in $x$,
and it determines a positive function $g_\V\in\G_X\cap C(X)$
such that
\begin{equation*}
r_{[-n,0]}(g_\V)\leq \left(1+R_n^{-1}\right)^2,
\end{equation*}
where
$$
R_n=\sum_{k=0}^\infty\prod_{j=0}^k\inf_{i\geq0}\bar r_{[-n,i+j]}(f_i)
\qquad(n\geq0).
$$
Moreover, $\mu\in\G(\V)$ if and only if $\mu$ is a $g_\V$-chain.
\end{theorem}

\begin{proof}
Given $\{f_n\}$,
let $\rho^{(n)}_k(\,\cdot\,,\,\cdot\,)$ be as in Section \ref{sec:rb}.
Note that
$$
\rho_0^{(n)}(x,y)\rho_0^{(n)}(y,x)
\leq\frac{\pi^\V_{[0,n]}\bigl(\,[x]_{\{0\}}\,\big|\,x\bigr)}
{\pi^\V_{[0,n]}\bigl(\,[x]_{\{0\}}\,\big|\,y\bigr)}.
$$
Then since $\V$ is positive and continuous,
Proposition \ref{prop:rb} (with $A=\N^-$) implies that
\begin{equation}\label{eq:uc}
\lim_{n\to\infty}\ \sup\,
\Bigl\{\pi^\V_{[0,n]}\bigl(\,[x]_{\{0\}}\,\big|\,x\bigr)
-\pi^\V_{[0,n]}\bigl(\,[x]_{\{0\}}\,\big|\,y\bigr)
\,:\,x,y\in X,\,x_{\N^-}=y_{\N^-}\Bigr\}
=0.
\end{equation}
Given $x\in X$, consider the sequences
$$
\sup_{y\in[x]_{\N^-}}
\pi^\V_{[0,n]}\bigl(\,[x]_{\{0\}}\,\big|\,y\bigr),
\qquad\inf_{y\in[x]_{\N^-}}
\pi^\V_{[0,n]}\bigl(\,[x]_{\{0\}}\,\big|\,y\bigr)
\qquad(n\geq0),
$$
which are respectively decreasing, increasing in $n$,
and hence convergent. It follows from \eqref{eq:uc} that
they converge to the same limit, and so
$$
g_\V(x)
=\lim_{n\to\infty}\pi^\V_{[0,n]}\bigl(\,[x]_{\{0\}}\,\big|\,x\bigr)
\qquad(x\in X)
$$
is well defined for all $x\in X$.
Furthermore,
$$
\sup_{x\in X}
\left|g_\V(x)-\pi^\V_{[0,n]}\bigl(\,[x]_{\{0\}}\,\big|\,x\bigr)\right|
\leq\sup_{x_{\N^-}=y_{\N^-}}
\Bigl\{\pi^\V_{[0,n]}\bigl(\,[x]_{\{0\}}\,\big|\,x\bigr)
-\pi^\V_{[0,n]}\bigl(\,[x]_{\{0\}}\,\big|\,y\bigr)\Bigr\},
$$
and so the convergence is uniform.
Thus, the limit is continuous,
and it follows from Proposition \ref{prop:g} that $g_\V\in\G_X$,
$g_\V$ is positive,
and $\mu\in\G(\V)$ if and only if $\mu$ is a $g_\V$-chain.

Applying Proposition \ref{prop:rb} again with $A=[-n,-1]$
(and $\{f_n\}$ as before) gives
\begin{align*}
\bar r_{[-n,0]}(g_\V)
&=\inf_{x_A=y_A}\ \lim_{m\to\infty}\ 
\frac{\pi^\V_{[0,m]}\bigl(\,[x]_{\{0\}}\big|\,x\bigr)}
{\pi^\V_{[0,m]}\bigl(\,[x]_{\{0\}}\big|\,y\bigr)}\\
&\geq\left(\frac{R_n}{1+R_n}\right)^2,
\end{align*}
as required.
\end{proof}

\begin{remark}
Proposition \ref{prop:g} and
Theorem \ref{thm:g} can easily be generalised to include non-invariant specifications $\V$,
in which case $g_\V$ is replaced by a sequence of continuous functions
$\{g_n\}_{n=-\infty}^\infty$ such that
$$
g_n(x)=\mu\bigl(\,[x]_{\{n\}}\,\big|\,\B_{(-\infty,n-1]}\bigr)(x)\qquad(n\in\Z).
$$
Note that in general, $g_n(x)\neq g_0(T^nx)$.
\end{remark}

The following result will be useful for estimating the variation of $g_\V$.

\begin{lemma}\label{lem:g}
Let $\V=\{\phi_\Lambda\}_{\Lambda\in\C}$
be a positive, continuous, $T$-invariant specification
with specification sequence $\{[0,n],f_n\}_{n=0}^\infty$. If
\begin{equation}\label{hyp2b}
\limsup_{n\to\infty}\,n^{\alpha-1}\prod_{i=0}^n\sup_{j\geq0}r_{(-\infty,i+j]}(f_j)
\leq K
\end{equation}
for some $0<\alpha\leq1,\,0<K<\infty$, then
$$
g_\V(x)=\lim_{n\to\infty}\pi^\V_{[0,n]}\bigl(\,[x]_{\{0\}}\,\big|\,x\bigr)
$$
is well defined, and
\begin{equation*}
\limsup_{n\to\infty}\,
\frac{\log r_{[-n,0]}(g_\V)}
{\Bigl(\log\sup_{i\geq0}r_{[-n,\infty)}(f_i)\Bigr)^\alpha}
\quad\leq\ 2\,\Gamma(\alpha)^{-1}K.
\end{equation*}
\end{lemma}

\begin{proof}
Note that \eqref{hyp2b} implies
$$
\sum_{n=0}^\infty\prod_{j=0}^n\inf_{i\geq0}\bar r_{(-\infty,i+j]}(f_i)=\infty,
$$
and so it follows from Theorem \ref{thm:g}
that $g_\V\in\G_X$ is well defined and
\begin{equation}\label{eq:gb}
\log r_{[-n,0]}(g_\V)
\leq2\log\left(1+R_n^{-1}\right)
\leq2R_n^{-1}
\qquad(n\geq 0),
\end{equation}
where
$$
R_n=\sum_{k=1}^\infty\prod_{j=0}^{k-1}\inf_{i\geq0}\bar r_{[-n,i+j]}(f_i)
$$
as before.
Note that
$$
R_n\geq\sum_{k=1}^\infty a_kt_n^k,
$$
where
$$
a_k=\prod_{j=0}^{k-1}\inf_{i\geq0}\bar r_{(-\infty,i+j]}(f_i),
\qquad t_n=\inf_{i\geq0}\bar r_{[-n,\infty)}(f_i).
$$
Note also that $a_n$ is decreasing, and the power series
$\sum_{k=1}^\infty a_kt^k$ is covergent for $0\leq t<1$.
For such series, we have the following Tauberian theorem (see \cite[p447]{F}):
for $0<\alpha<\infty$,
$$
\sum_{k=1}^\infty a_kt^k\ \sim\ (1-t)^{-\alpha}\quad\text{as }t\nearrow1
\iff a_n\ \sim\ \Gamma(\alpha)^{-1}n^{\alpha-1}\quad\text{as }n\to\infty.
$$
Thus, if
$$
\limsup_{n\to\infty}\,n^{\alpha-1}a_n^{-1}\leq K,
$$
where $0<\alpha\leq1$, $0<K<\infty$, then
$$
\lim_{t\to 1^-}\left((1-t)^\alpha \sum_{k=0}^\infty a_kt^k\right)^{-1}
\leq \Gamma(\alpha)^{-1}K,
$$
and since $1-t\sim|\log t|$ as $t\nearrow1$,
and $t_n\nearrow1$ as $n\to\infty$, it follows from \eqref{eq:gb} that
\begin{align*}
\limsup_{n\to\infty}\,\Bigl(\log\sup_{i\geq0}r_{[-n,\infty)}(f_i)\Bigr)^{-\alpha}
\log r_{[-n,0]}(g_\V)
&=2\limsup_{n\to\infty}\,|\log t_n|^{-\alpha}\log r_{[-n,0]}(g_\V)\\
&\leq 2\,\Gamma(\alpha)^{-1}K,
\end{align*}
as required.
\end{proof}

The next result follows from the results and techniques of
Johansson, \"Oberg and Pollicott in \cite{JOP},
although it is not formulated there; for an explicit proof, see \cite{H3}.

\begin{theorem}\label{thm:c}
Let $g\in\G$ be positive and continuous.
If
\begin{equation}\label{eq:hyp5}
\lim_{n\to\infty}\sum_{i=\lceil\lambda^{n-1}\rceil}^{\lceil\lambda^n\rceil}
(\log r_{[-i,0]}(g_\V))^2=0
\end{equation}
for some $\lambda>1$,
then there is a unique $g$-chain $\mu\in M(X)$, which is $T$-invariant and Bernoulli.
\end{theorem}

Note that if \eqref{eq:hyp5} holds for some $\lambda>1$,
then it holds for all such $\lambda$.

\section{Proof of main results}

\begin{theorem}\label{thm:bern1}
Let $\V=\{\phi_\Lambda\}_{\Lambda\in\C}$
be a positive, continuous, $T$-invariant specification on $S^\Z$
with specification sequence $\{[0,n],f_n\}_{n=0}^\infty$.
If, for some $0<\alpha\leq1$,
\begin{equation}\label{hyp2a}
\prod_{i=0}^n\sup_{j\geq0}r_{(-\infty,i+j]}(f_j)
=O\bigl(n^{1-\alpha}\bigr),
\end{equation}
and
\begin{equation}\label{hyp4a}
\sum_{i=2^n}^{2^{n+1}}
\Bigl(\log\sup_{j\geq0}r_{[-i,\infty)}(f_j)\Bigr)^{2\alpha}=o(1),
\end{equation}
then $\V$ has a unique Gibbs state $\mu$, which is $T$-invariant,
and $(T,\mu)$ is Bernoulli.
\end{theorem}

\begin{proof}
It follows from \eqref{hyp2a},
Theorem  \ref{thm:g} and Lemma \ref{lem:g}
that there is a positive, continuous $g_\V\in\G_X$ such that
any Gibbs state for $\V$ is a $g_\V$-chain, and
\begin{equation*}
\log r_{[-n,0]}(g_\V)
=O\left(\Bigl(\log\sup_{i\geq0}r_{[-n,\infty)}(f_i)\Bigr)^\alpha\right).
\end{equation*}
The result now follows from \eqref{hyp4a} and Theorem \ref{thm:c}.
\end{proof}

Theorem \ref{thm:bern} follows immediately from Theorem \ref{thm:bern1}
and \eqref{eq:fs}--\eqref{eq:f+var}.
If the specification is described in terms of an interaction potential
$\Phi=\{\Phi_\Lambda\}_{\Lambda\in\C}$,
then it follows from \eqref{eq:ipf+} and \eqref{eq:ipr}
that \eqref{hyp2a} and  \eqref{hyp4a} hold provided
there are constants $K>0$ and $0<\alpha\leq1$ such that
$$
\sum_{i=0}^n
\sum_{\substack{\Lambda\ni0\\\Lambda\cap[i+1,\infty)\neq\emptyset}}
\var(\Phi_\Lambda)
\ \leq\ (1-\alpha)\log n+K\qquad(n\geq0),
$$
and
$$
\lim_{n\to\infty}\sum_{i=2^n}^{2^{n+1}}
\Biggl(\sum_{\substack{\Lambda\ni0\\\Lambda\cap[i+1,\infty)\neq\emptyset}}
\var(\Phi_\Lambda)\Biggr)^{2\alpha}=0.
$$
This will be the case if
$$
\sum_{\substack{\Lambda\ni0\\\Lambda\cap[n+1,\infty)\neq\emptyset}}
\var(\Phi_\Lambda)
\leq(1-\alpha)n^{-1}+a_n
\qquad(n\geq1)
$$
for some $1/2<\alpha\leq1$ and $\sum a_n<\infty$,
that is, if \eqref{hyp3} holds with $\beta=1-\alpha$;
thus Theorem \ref{thm:bern2} follows.

Since
$$
\sum_{n=0}^\infty\sum_{\substack{\Lambda\ni0\\\Lambda\cap[n+1,\infty)\neq\emptyset}}
\var(\Phi_\Lambda)
\leq\sum_{\Lambda\ni0}\text{diam}(\Lambda)\|\Phi_\Lambda\|,
$$
Theorem \ref{thm:bern2} includes potentials satisfying \eqref{eq:h2}.
Comparison with \eqref{eq:h3} and \eqref{eq:h4} is less straightforward;
in some cases, its hypotheses are weaker, but not always.
Note however that Theorem \ref{thm:bern2} guarantees a unique Gibbs state
(absence of symmetry breakdown), whereas \eqref{eq:h3} only guarantees uniqueness
of $T$-invariant Gibbs states, and \eqref{eq:h4} does not include the Bernoulli property. 

By way of comparison,
consider the Ising model on $\{-1,1\}^\Z$ with pair interactions
\begin{equation}\label{ising}
\Phi_{\{i,j\}}(x)=-\frac{\beta}{2}J(|i-j|)x_ix_j,
\end{equation}
where $\beta>0$ (the inverse temperature) is a constant and
$|J(n)|\leq n^{-\alpha}$ for some $\alpha>1$.
Since Berbee's result applies to non-invariant specifications
and does not depend on any group structure, there is no loss of generality
in assuming that $I=\Z$ and
\begin{equation*}
\Lambda_{2k}=[-k,k],\qquad\Lambda_{2k+1}=[-k,k+1]
\qquad(k\geq0)
\end{equation*}
(as in \eqref{eq:li}), other than the assumption that
$|\Lambda_n\setminus\Lambda_{n-1}|=1$.

The corresponding specification sequence $\{\Lambda_n,f_n\}$ is given by
\begin{equation*}
f_n(x)
=\exp\left(\frac{\beta}{2}\sum_{j=1}^\infty
|i_n-i_{n+j}|^{-\alpha}x_{i_n}x_{i_{n+j}}\right)
\qquad(n\geq 0),
\end{equation*}
where $i_n$ is the unique element of $\Lambda_n\setminus\Lambda_{n-1}$.
Note that
$\inf_{i\geq 0}\bar r_{\Lambda_{i+k}}(f_i)=\bar r_{\Lambda_k}(f_0)$,
\begin{equation*}
\begin{aligned}
\bar r_{\Lambda_{2k}}(f_0)&=
\exp\left(-2\beta\sum_{j=k+1}^\infty j^{-\alpha}\right),\\
\bar r_{\Lambda_{2k+1}}(f_0)&=
\exp\left(-2\beta\sum_{j=k+1}^\infty j^{-\alpha}
+\beta(k+1)^{-\alpha}\right),
\end{aligned}
\end{equation*}
and
$$
r_{(-\infty,n]}(f_0)=r_{[-n,\infty)}(f_0)
=\exp\left(\beta\sum_{j=n+1}^\infty j^{-\alpha}\right).
$$
Thus, for $\alpha>2$, Berbee's condition and Theorem \ref{thm:bern} both
give uniqueness and the Bernoulli condition for all $\beta$ (as do the results in \cite{CQ} and \cite{R68}),
whereas for $\alpha>2$, no conclusion can be drawn.
For the borderline case $\alpha=2$,
no conclusion can be drawn from the results in \cite{CQ} and \cite{R68},
but Berbee's condition gives
uniqueness for $\beta\leq\frac{1}{4}$;
from Theorem \ref{thm:bern} we get uniqueness and the Bernoulli property
for $\beta<\frac{1}{2}$.

Note that a better uniqueness result for the case $J(n)=n^{-2}$ is obtained in \cite{ACCN}.
In this case, since  $J(n)\geq0$ for all $n$, the interaction potential is ferromagnetic,
and as is well known, there is a critical inverse temperature $\beta_c>0$
such that there is a unique Gibbs state when $0\leq\beta<\beta_c$,
but non-unique $T$-invariant Gibbs states for $\beta>\beta_c$.
It is shown in \cite{ACCN} that $\beta_c>1$ (with non-uniqueness at $\beta_c$).
In particular, there is a unique Gibbs state when $\beta\leq1$;
moreover, it follows from results in \cite{H1} that this Gibbs state is $T$-invariant and Bernoulli.

\section{Uniqueness without the Bernoulli property}

A recent result of
Berger,  Conache,  Johansson and \"Oberg \cite{BCJO}
proves uniqueness of the $g$-measure for $g\in\G$ such that
\begin{equation}\label{eq:gc}
\limsup_{n\to\infty}\,n^{1/2}\log r_{[-n,0]}(g) \,< \,2.
\end{equation}
Using this result instead of  Theorem \ref{thm:c} gives uniqueness of the $T$-invariant Gibbs state
under hypotheses which are slightly weaker in some cases than those of Theorem \ref{thm:bern1},
at the expense of the Bernoulli property.

\begin{theorem}
Let $\V=\{\phi_\Lambda\}_{\Lambda\in\C}$
be a positive, continuous, $T$-invariant specification on $S^\Z$
with specification sequence $\{[0,n],f_n\}_{n=0}^\infty$.
If, for some $0<\alpha\leq1$ and $K>0$,
\begin{equation}\label{hyp2ai}
\limsup_{n\to\infty}\,
n^{\alpha-1}\,\prod_{i=0}^n\sup_{j\geq0}r_{(-\infty,i+j]}(f_j)
\leq K,
\end{equation}
and
\begin{equation}\label{hyp4ai}
\limsup_{n\to\infty}\,n^{1/2}
\Bigl(\log\sup_{j\geq0}r_{[-i,\infty)}(f_j)\Bigr)^\alpha\,
<\,K^{-1}\Gamma(\alpha),
\end{equation}
then $\V$ has a unique $T$-invariant Gibbs state.
\end{theorem}

\begin{proof}
It follows from \eqref{hyp2ai} and Theorem  \ref{thm:g} that
any Gibbs state for $\V$ is a $g_\V$-chain, where
$$
g_\V(x)=\lim_{n\to\infty}\pi^\V_{[0,n]}\bigl(\,[x]_{\{0\}}\,\big|\,x\bigr),
$$
and from Lemma \ref{lem:g} that $g_\V$ satisfies \eqref{eq:gc} provided \eqref{hyp4ai} holds;
thus the result follows.
\end{proof}

\begin{remark}
Note that there is not necessarily a unique Gibbs state,
unless $\V$ is also attractive,
in which case the unique $T$-invariant Gibbs state is also the unique Gibbs state,
and is a Bernoulli measure for $T$ (see \cite{H1}).
\end{remark}

If the specification is described in terms of an interaction potential
$\Phi=\{\Phi_\Lambda\}_{\Lambda\in\C}$,
then it follows from \eqref{eq:ipf+} and \eqref{eq:ipr}
that $\Phi$ has a unique $T$-invariant Gibbs state provided
there are constants $K>0$ and $0<\alpha\leq1$ such that
$$
\sum_{i=0}^n
\sum_{\substack{\Lambda\ni0\\\Lambda\cap[i+1,\infty)\neq\emptyset}}
\var(\Phi_\Lambda)
\ \leq\ (1-\alpha)\log n+\log K\qquad(n\geq0),
$$
and
$$
\limsup_{n\to\infty}\,n^{1/2}\,
\Biggl(\sum_{\substack{\Lambda\ni0\\\Lambda\cap[n+1,\infty)\neq\emptyset}}
\var(\Phi_\Lambda)\Biggr)^\alpha<K^{-1}\Gamma(\alpha).
$$
Since $\Gamma(1/2)=\sqrt{\pi}$, this will be the case if, for example,
$$
\sum_{\substack{\Lambda\ni0\\\Lambda\cap[n,\infty)\neq\emptyset}}
\var(\Phi_\Lambda)
\leq\, \frac{1}{2n}
\qquad(n\geq1).
$$
(Note that this does not imply \eqref{hyp4a} holds.)

\end{document}